\documentclass[12pt, reqno]{amsart}

\usepackage{fullpage}   
\usepackage{amssymb}    
\usepackage{url}

\theoremstyle{plain}
\newtheorem*{main}{Theorem}
\newtheorem{theorem}{Theorem}
\newtheorem{lemma}{Lemma}
\theoremstyle{remark}
\newtheorem*{acknowledgment}{Acknowledgment}

\numberwithin{equation}{section}

\newcommand{\seclabel}[1]{\label{sec:#1}}   
\newcommand{\eqnlabel}[1]{\label{eqn:#1}}   

\newcommand{\secref}[1]{\ref{sec:#1}}   
\newcommand{\eqnref}[1]{\eqref{eqn:#1}} 

\newcommand{\by}[1]{\overset{\eqnref{#1}}=}  
\newcommand{\byx}[1]{\overset{#1}=}          

\newcommand{\Eid}{\mathbf{E}}

\title{An Elegant 3-Basis for Inverse Semigroups}

\author{Jo\~{a}o Ara\'{u}jo}
\author{Michael Kinyon}

\address[Ara\'{u}jo]{Centro de \'{A}lgebra \\
Universidade de Lisboa \\
1649-003 Lisboa \\ Portugal\\
and\\
Universidade Aberta\\
1269--001 Lisboa \\ Portugal}
\email{\url{jaraujo@ptmat.fc.ul.pt}}

\address[Kinyon]{Department of Mathematics \\
University of Denver \\ 2360 S Gaylord St \\ Denver, Colorado 80208 USA}
\email{\url{mkinyon@math.du.edu}}

\begin{document}

\begin{abstract}
It is well known that in every inverse semigroup the binary operation and the  unary operation of inversion satisfy the following three identities:
\[
\quad x=(xx')x \qquad
\quad (xx')(y'y)=(y'y)(xx') \qquad
\quad (xy)z=x(yz'')\,.
\] The goal of this note is to prove the converse, that is, we prove that an algebra of type $\langle 2,1\rangle$ satisfying these three identities is an inverse semigroup and the unary operation coincides with the usual inversion on such semigroups.
\end{abstract}

\maketitle

\section{Introduction}
\seclabel{intro}

In the language of a binary operation $\cdot$ and a unary operation ${}'$, a set of $n$ independent identities is an $n$-basis
for inverse semigroups, if those identities define the variety of inverse semigroups considered as algebras $(S,\cdot,{}')$
of type $\langle 2,1\rangle$, where the unary operation coincides with the natural inversion.
Denoting by $x'$ the inverse of an element $x$ in an inverse semigroup, we then have $x=(xx')x$ (as inverse semigroups are
regular semigroups) and $(xx')(y'y)=(y'y)(xx')$ (as both $xx'$ and $y'y$ are idempotents, and idempotents commute in inverse
semigroups). Thus we might be tempted to think that the following identities provide a $3$-basis for inverse semigroups:
\begin{equation}
\eqnlabel{candidates}
x=(xx')x,  \qquad (xx')(y'y)=(y'y)(xx') \qquad  \text{and}\qquad (xy)z=x(yz)\,.
\end{equation}
However, for $S=\{0,1\}$ with $xy=0$, except for $11=1$, and defining $x'=1$, we have the previous identities satisfied,
but $0'\neq 0'00'$ and hence $'$ does not coincide with the natural inversion in $(S,\cdot)$.

B.M. Schein \cite{Schein} repaired the {\em defect} of \eqnref{candidates} by adjoining two additional identities: $x''=x$ and $(xy)'=y'x'$. The resulting set of five identities indeed provides a $4$-basis for inverse semigroups. (The identity $(xy)'=y'x'$ is dependent upon the others, and hence can be discarded. However it is worth observing that in the same paper Schein also provided a $5$-basis using $xx'x'x=x'xxx'$ instead of $xx'y'y=y'yxx'$; see \cite[Theorem 1.6]{Schein} and
\cite[p. 15, Ex. 20(b)]{higgins}.) Therefore the natural question to ask would be: \emph{is it possible to find a 3-basis for inverse semigroups?} This question was first answered in the affirmative in \cite{AM}, but the $3$-basis given there requires an extremely complicated proof (it is still an open problem to provide a reasonable proof for that result).

The aim of this note is to repair \eqnref{candidates} by providing an easy, transparent and
\emph{elegant} $3$-basis for inverse semigroups.

\begin{main}
Let $(S,*,')$ be an algebra of type $\langle 2,1\rangle$. Then this algebra is an inverse semigroup and the unary operation coincides with the usual inversion on such semigroups if and only if
\[
(\Eid_1)\quad x=(xx')x, \qquad
(\Eid_2)\quad (xx')(y'y)=(y'y)(xx'), \qquad
(\Eid_3)\quad (xy)z=x(yz'')\,.
\]
\end{main}

\section{Proof of the Theorem}
\seclabel{Proof}

In this section we prove that the identities ($\Eid_1$)--($\Eid_3$) imply Schein's $4$-basis for inverse semigroups. As the converse is obvious, the equivalence of the two bases will follow.

Throughout this section let $(S,\cdot,{}')$ be an algebra of type $\langle 2,1\rangle$ satisfying ($\Eid_1$)--($\Eid_3$). We start by proving a few handy identities.

\begin{lemma} The following identities hold.
\begin{align}
x'x''&= x'x  \eqnlabel{lemma1} \\
(xy')y&=x(y'y)  \eqnlabel{lemma5}\\
x & =x(x'x)  \eqnlabel{lemma3b}\\
x''&= (x''x')x = x''(x'x) \eqnlabel{lemma6}\\
x'''x&=x'''x''=x''' x^{(4)}  \eqnlabel{lemma50}
\end{align}
\end{lemma}

\begin{proof}
Firstly, for \eqnref{lemma1}, we have
\[
x'x'' \byx{(\Eid_1)} x'[(x''x''')x'']
\byx{(\Eid_3)} [x'(x''x''')]x \byx{(\Eid_3)} [(x'x'')x']x \byx{(\Eid_1)} x'x\,.
\]

Next, for \eqnref{lemma5}, we compute $(xy')y \byx{(\Eid_3)} x(y'y'') \by{lemma1} x(y'y)$.

Regarding \eqnref{lemma3b}, we have $x(x'x) \by{lemma5} (xx')x \byx{(\Eid_1)} x$.

Then for \eqnref{lemma6}, we compute
$x'' \by{lemma3b} x''(x'''x'') \byx{(\Eid_3)} (x''x''')x \by{lemma1} (x''x')x \by{lemma5} x''(x'x)$.

Finally, for \eqnref{lemma50}, we have
\[
x'''x \by{lemma3b} [x'''(x''''x''')]x \byx{(\Eid_3)} x'''[(x''''x''')x'']
\by{lemma6} x''' x'''' \by{lemma1} x''' x''\,.
\]
\end{proof}

The next two lemmas are the key tools in the proof that the identities ($\Eid_1$)--($\Eid_3$)
imply $x''=x$.

\begin{lemma}
\label{396}
 $(x'x)x'''=x'''$.
\end{lemma}

\begin{proof}
 We start with two observations. Firstly,
 as
\[
[x(y''' y)]y' \byx{(\Eid_3)} x[(y'''y)y''']
\by{lemma50} x[(y'''y'''')y'''] \byx{(\Eid_1)} xy'''\,,
\]
we have
\begin{equation}
\eqnlabel{136}
(x(y''' y))y' = xy'''\,.
\end{equation}

Secondly,
\[
(x'x)(x'''x) \by{lemma50} (x'x)(x'''x'''')
\byx{(\Eid_2)} (x'''x'''')(x'x) \by{lemma50} (x'''x'')(x'x)
\by{lemma5} [(x'''x'')x']x \by{lemma6} x'''x\,,
\]
so that
\begin{align}
(x'x)(x'''x)&=x'''x.  \eqnlabel{148}
\end{align}

Now we have all we need to prove the lemma.
\[
 x'''\by{lemma6}(x'''x'')x'\by{lemma50}(x'''x)x'\by{148}[(x'x)(x'''x)]x'\by{136}(x'x)x'''.
\]
\end{proof}

\begin{lemma}\label{lemma16}
 $(xy)z'=x(yz')$.
\end{lemma}

\begin{proof}
We start by proving that
\begin{equation}
\eqnlabel{455}
x'''=x'\,.
\end{equation}
In fact we have $xx' \by{lemma3b} [x(x'x)]x' \byx{(\Eid_3)} x[(x'x)x'''] = xx'''$, using
Lemma \ref{396} in the last equality. Thus
\begin{equation}
\eqnlabel{452}
xx'''=xx'\,.
\end{equation}
Now, by Lemma \ref{396},
\[
 x'''=(x'x)x''' \by{lemma1} (x'x'')x'''
\byx{(\Eid_3)} x'(x''x^{(5)}) \by{452} x'(x''x''') \byx{(\Eid_3)}
(x'x'')x' \byx{(\Eid_1)} x'\,.
\]

Replacing $z$ by $z'$ in ($\Eid_3$), we get
\[
 (xy)z'=x(yz''')=x(yz')\,,
\]
where the last equality follows from  \eqnref{455}. The lemma is proved.
\end{proof}

We have everything we need to prove our main result.

\begin{theorem}
The identities \emph{(}$\Eid_1$\emph{)}--\emph{(}$\Eid_3$\emph{)} imply $x'' = x$ and the associative law.
\end{theorem}

\begin{proof}
First, we have
\begin{align*}
x''x' &\by{lemma6} [(x''x')x]x' = (x''x')(xx') \byx{(\Eid_2)} (xx')(x''x')\\
&= [(xx')x'']x' = [x(x'x'')]x' = x[(x'x'')x'] \byx{(\Eid_1)} xx'\,,
\end{align*}
where we have used Lemma \ref{lemma16} in the unlabeled equalities. Thus
\begin{equation}
\eqnlabel{hmph}
x''x' = xx'\,.
\end{equation}

Now $x'' \by{lemma6} (x''x')x \by{hmph} (xx')x \byx{(\Eid_1)} x$, as claimed.

Associativity now follows easily: $(xy)z \byx{(\Eid_1)} x(yz'') = x(yz)$.
\end{proof}

\section{Other Sets of Axioms}

It is natural to ask how sensitive the axioms ($\Eid_1$)--($\Eid_3$) are
to certain modifications, such as shifting the parentheses in ($\Eid_1$)
or changing the placement of the double inverse in ($\Eid_3$).

If, for instance,
we leave ($\Eid_2$) intact, replace ($\Eid_1$) with $x(x'x) = x$ and replace ($\Eid_3$) with
$(x''y)z = x(yz)$, then we obtain a set of identities which are dual to
($\Eid_1$)--($\Eid_3$). By an argument dual to that in \S\secref{Proof},
this set of identities is another $3$-basis for inverse semigroups.

Thus to dispense with these sorts of obvious dualities, we will assume that
both ($\Eid_1$) and ($\Eid_2$) are left intact, and consider only alternative placement
of the double inverse in ($\Eid_3$). Using \textsc{Prover9}, we found that each of
the following identities can substitute for ($\Eid_3$) to give another $3$-basis for
inverse semigroups:
\begin{align*}
(xy)z &= x''(yz) \hspace{2cm}  (xy)z = x(y''z) \\
x(yz) &= (xy'')z   \hspace{2cm}   x(yz) = (xy)z''.
\end{align*}

The remaining possibility, $x(yz) = (x''y)z$, does not work. Using \textsc{Mace4}, we found
the counterexample given by the following tables. It satisfies ($\Eid_1$), ($\Eid_2$) and
$x(yz) = (x''y)z$, but the binary operation is not associative
($(0\cdot 0)\cdot 0 = 1\cdot 0 = 7\neq 6 = 0\cdot 1 = 0\cdot (0\cdot 0)$), and
the unary operation clearly fails to satisfy $x'' = x$.

\begin{table}[htb]
\centering
\begin{tabular}{r|cccccccccccc}
$\cdot$ & 0 & 1 & 2 & 3 & 4 & 5 & 6 & 7 & 8 & 9 & 10 & 11\\
\hline
    0 & 1 & 6 & 5 & 7 & 3 & 8 & 4 & 2 & 0 & 4 & 4 & 4 \\
    1 & 7 & 2 & 6 & 0 & 8 & 4 & 5 & 1 & 3 & 5 & 5 & 5 \\
    2 & 5 & 8 & 3 & 6 & 1 & 7 & 0 & 4 & 2 & 0 & 0 & 0 \\
    3 & 8 & 0 & 7 & 4 & 6 & 2 & 1 & 3 & 5 & 1 & 1 & 1 \\
    4 & 3 & 7 & 1 & 8 & 5 & 6 & 2 & 0 & 4 & 2 & 2 & 2 \\
    5 & 6 & 4 & 8 & 2 & 7 & 0 & 3 & 5 & 1 & 3 & 3 & 3 \\
    6 & 0 & 1 & 2 & 3 & 4 & 5 & 6 & 7 & 8 & 6 & 6 & 6 \\
    7 & 4 & 3 & 0 & 5 & 2 & 1 & 7 & 8 & 6 & 7 & 7 & 7 \\
    8 & 2 & 5 & 4 & 1 & 0 & 3 & 8 & 6 & 7 & 8 & 8 & 8 \\
    9 & 0 & 1 & 2 & 3 & 4 & 5 & 6 & 7 & 8 & 9 & 10 & 6 \\
    10 & 0 & 1 & 2 & 3 & 4 & 5 & 6 & 7 & 8 & 10 & 9 & 6 \\
    11 & 0 & 1 & 2 & 3 & 4 & 5 & 6 & 7 & 8 & 6 & 6 & 11
\end{tabular}

\bigskip

\begin{tabular}{r|cccccccccccc}
${}'$ & 0 & 1 & 2 & 3 & 4 & 5 & 6 & 7 & 8 & 9 & 10 & 11 \\
\hline
 & 1 & 2 & 3 & 4 & 5 & 0 & 6 & 8 & 7 & 9 & 10 & 11
\end{tabular}
\end{table}

\section{Problem}

\emph{Does there exist a $2$-basis for inverse semigroups?}

\medskip

We guess that the answer is no.

\begin{acknowledgment}
We are pleased to acknowledge the assistance of the automated deduction tool
\textsc{Prover9} and the finite model builder \textsc{Mace4}, both developed by
McCune \cite{McCune}.

The first author was partially supported by FCT and FEDER, Project POCTI-ISFL-1-143 of Centro de Algebra da Universidade de Lisboa, and by FCT and PIDDAC through the project PTDC/MAT/69514/2006.
\end{acknowledgment}

\end{document}